\documentclass[10pt]{article}
    
\usepackage{amsmath,amsfonts,amsthm,amscd,amssymb,graphicx}
\usepackage{mathabx}

\usepackage[utf8]{inputenc}

\usepackage{subfigure}
\usepackage{xcolor}

\numberwithin{equation}{section}

\usepackage{enumitem}
\usepackage{hyperref}

\usepackage{cases}

\newtheorem{theorem}{Theorem}[section]

\newtheorem{lemma}[theorem]{Lemma}

\newtheorem{proposition}[theorem]{Proposition}

\newtheorem{remark}[theorem]{Remark}

\newcommand{\RR}{\mathbb{R}}

\makeatletter
\newcommand{\labitem}[2]{%
	\def\@itemlabel{\textbf{#1}}
	\item
	\def\@currentlabel{#1}\label{#2}}
\makeatother

\begin{document}

\title{Phase mixing estimates for the nonlinear Hartree equation of infinite rank} 

\author{Chanjin You\footnotemark[1]
}

\maketitle

\footnotetext[1]{Penn State University, Department of Mathematics, State College, PA 16802. Email: cby5175@psu.edu.}


\begin{abstract}

In this paper, we prove the phase mixing estimates for the density and its derivatives associated with the nonlinear Hartree equation around certain translation-invariant equilibria. Given a defocusing short-range interaction potential, we provide a precise criterion for the Penrose--Lindhard stability based on the marginal of the equilibrium. For linearly stable equilibria, pointwise decay estimates of the Green function associated with the linearized operator in Fourier space are established. The proof of phase mixing estimates is obtained through a nonlinear iterative scheme. An alternative proof of scattering is also provided.

\end{abstract}


\section{Introduction}
A system of infinitely many quantum particles in $\mathbb{R}^d$, $d \ge 3$, can be described by the nonlinear Hartree equation
\begin{equation}
	\label{Hartree}
	\begin{cases}
		\begin{aligned}
			&i \partial_{t} \gamma = [ - \Delta + w \star_x \rho_\gamma, \, \gamma ] \\
			&\gamma_{\vert_{t=0}} = \gamma_{0}
		\end{aligned}
	\end{cases}
\end{equation}
where $\gamma(t)$ is a one-particle density operator and $\rho_{\gamma}(t,x)$ is a density function associated with $\gamma(t)$. Namely, $\gamma(t)$ is a nonnegative  bounded operator on $L^2(\mathbb{R}^d)$, and its density is defined by
	\begin{equation}\label{def:density}
		\rho_{\gamma}(t,x) = \gamma(t,x,x),
	\end{equation}
where $\gamma(t,x,y)$ is the integral kernel of $\gamma(t)$. The two-body interaction potential $w$ is a given finite positive Borel measure on $\mathbb{R}^d$, such that $\widehat{w}$ is radial, $\widehat{w}(k) \ge 0$, $\widehat{w}(0) <\infty$, and $\widehat{w}'(k) \le 0$.
It covers short-range potentials, i.e., $w \in L^1$, as well as singular measures like the Dirac delta measure $w= \delta$.
The convolution term $w \star_x \rho_{\gamma}$ is defined by
	\[
		(w\star_x \rho_{\gamma})(x) = \int_{\mathbb{R}^d} \rho_{\gamma}(x-y) \, dw(y),
	\]
which captures the mean-field interaction among particles.

In the case when $\gamma(t)$ is of finite rank with the spectral decomposition given by
\[
	\gamma(t) = \sum_{j=1}^{N} \lvert u_{j}(t) \rangle \langle u_{j}(t) \rvert
\]
for some finite $N$ and orthonormal functions $u_{j}(t)$ in $L^2 (\mathbb{R}^d)$, the Cauchy problem \eqref{Hartree} in the density operator form reduces to a system of coupled Hartree equations
\begin{equation}
	\label{Hartree2}
	\begin{cases}
		\begin{aligned}
			&i \partial_{t} u_{j} = \Big(-\Delta + w \star_x \sum_{k=1}^{N} |u_{k}|^2 \Big) u_{j} \\
			&u_{j {\vert_{t=0}}} = u_{j,0}
		\end{aligned}
	\end{cases}
\end{equation}
for $1 \le j \le N$, noting that $$\rho_{\gamma}(t,x) = \sum_{j=1}^{N}  |u_{j}(t,x)|^2.$$ Since $\int_{\mathbb{R}^d} \rho_{\gamma}(t) \, dx = N$ can be interpreted as the number of particles, it is evident that \eqref{Hartree} provides a natural generalization of this coupled system to infinite quantum systems.

The Cauchy problem \eqref{Hartree} for trace class initial data was studied in \cite{Bove1974, Bove1976, Chadam1976, Zagatti1992}. Global well-posedness and small data scattering for the general Kohn-Sham equation was proved in \cite{Pusateri2021}. Modified scattering for small data solutions to \eqref{Hartree} with the Coulomb potential in three dimensions is established in the companion paper \cite{Nguyen2025a}.
Meanwhile, Lewin and Sabin \cite{Lewin2015} first addressed the Hartree equation for density operators outside the trace class. They considered the global-in-time Cauchy problem for perturbations near a translation-invariant steady state $f = f(-\Delta)$ with finite constant density $\rho_{f}= \int_{\mathbb{R}^d} f(|p|^2) dp$, where the two-body interaction potential $w$ is assumed to be sufficiently smooth and decay rapidly. To be specific, they considered the equation for the perturbation
\begin{equation}
	\label{Hartree-perturbation}
	\begin{cases}
		\begin{aligned}
			&i \partial_{t} \gamma = [ - \Delta + w \star_x \rho_\gamma, \, \gamma + f ] \\
			&\gamma_{\vert_{t=0}} = \gamma_{0}.
		\end{aligned}
	\end{cases}
\end{equation}
The asymptotic stability of such equilibria in $d=2$ was shown in \cite{Lewin2014}, later extended to $d\ge 3$ in \cite{Chen2018}. These results were extended to singular interaction potentials (e.g., Dirac delta potential), and a larger class of steady states (e.g., Fermi gas at zero temperature), in \cite{Chen2017, Hadama2025b, Hadama2025c}. The linear analysis for Coulomb potential case, which is long-range, was established in \cite{Nguyen2025}. For an equivalent formulation of \eqref{Hartree-perturbation} using random fields, refer to \cite{Suzzoni2015, Collot2020, Collot2022}.

In this paper, we are interested in the asymptotic stability of the density operator near a class of translation-invariant equilibria $f = f(-\Delta)$ with finite regularity. We consider \eqref{Hartree-perturbation} for $t \ge 0$ where $w$ is a finite positive Borel measure on $\mathbb{R}^d$ whose Fourier transform $\widehat{w}$ is radial and satisfies $\widehat{w}(k) \ge 0$, $\widehat{w}(0) <\infty$, and $\widehat{w}'(k) \le 0$. Following the framework in \cite[Lemma~2.1]{Nguyen2025}, we use the Fourier-Laplace transform to obtain the resolvent equation
\begin{equation}\label{eq:resolvent-intro}
	\widetilde{\rho}_{\gamma}(\lambda, k) = \frac{1}{D(\lambda, k)} \widetilde{S}(\lambda,k)
\end{equation}
where the dispersion relation, known as the Lindhard dielectric function, is given by
\[
	D(\lambda, k)= 1 + i \widehat{w}(k) \int_{\mathbb{R}^d} \frac{f(|p|^2) - f(|k-p|^2) }{\lambda - 2i k \cdot p + i |k|^2} \, dp
\]
and $\widetilde{S}(\lambda, k)$ is the Fourier-Laplace transform of the source term
\[
	S(t,x) = \rho\left[ e^{it\Delta} \gamma_{0} e^{-it\Delta} -i \int_{0}^{t} e^{i(t-s)\Delta} [w  \star_x \rho_{\gamma}, \gamma ](s) e^{i(s-t)\Delta} \, ds \right](x).
\]
For the short-range nature of the potential, it is expected that certain classes of equilibria do not induce any oscillatory modes of their Lindhard dielectric functions $D(\lambda, k)$, which implies linear stability.
While previous works 
provided only sufficient conditions, this work establishes a precise criterion \ref{assumption:stability} for the nonexistence of oscillatory modes, see Proposition \ref{prop:nozero3}. This criterion is achieved through a detailed analysis of the Fourier-Laplace symbol $D(\lambda, k)$.
Some physically important equilibria including a Fermi gas at zero temperature in $d \ge 3$ fail to satisfy \ref{assumption:stability}, and their asymptotic stability is beyond the scope of this paper but will be of future interest to us. 

Another contribution of this paper is the pointwise time decay of the density function and its derivatives, known as the \textit{phase mixing} estimates, which have not been established for the Hartree equation at a nonlinear level. We prove that the $L^{\infty}_x$ decay of the density $\partial_x^{n} \rho_{\gamma}(t)$ associated with \eqref{Hartree-perturbation} is $t^{-d-n}$. This is optimal in the sense that it decays at the same rate as the density of free Hartree dynamics and gains an additional $ t^{-1}$ decay for each spatial derivative. Propagation of such estimates is important in view of previous works on \textit{Landau damping} for the classical case \cite{Mouhot2011, Hwang2011, Bedrossian2018, HanKwan2021, Nguyen2020}, in the context of understanding the long time behavior of classical particles. In this work, we establish phase mixing for the nonlinear Hartree equation under the Penrose-Lindhard stability condition. It follows from the resolvent equation \eqref{eq:resolvent-intro} that
\[
	\widehat{\rho}_{k}(t) = \int_{0}^{t} \widehat{G}_{k}(t-s) \widehat{S}_{k}(s) \, ds.
\]
Unlike the Coulomb case \cite{Nguyen2025}, the Penrose type condition ensures that the Green function in Fourier space does not have an oscillatory component, and the regular part exhibits rapid decay in $\langle kt \rangle$, see Proposition \ref{prop:Green}. Combining this with phase mixing estimates for the source term, we are able to close the bootstrap argument, see Section \ref{sec:nonlinear}. The choice of coordinates $\partial_k - \partial_p$ is crucial in the nonlinear estimate, noting that $(\partial_k - \partial_p) \widehat{V}(k+p) =0$ for $V= w \star_x \rho_{\gamma}$.

Finally, using the phase mixing estimates for $\rho_{\gamma}(t)$, we give an alternative proof of scattering of the solution $\gamma(t)$. This also demonstrates that the large time behavior of $\gamma(t)$ is well-approximated by the free Hartree dynamics.

\subsection{Equilibria and interaction potentials}\label{subsec:eq}
In this paper, $f = f(-\Delta)$ is a translation-invariant equilibrium of the Hartree equation with finite constant density $\rho_{f} =\int_{\mathbb{R}^d} f(|p|^2) \, dp$ and $w$ is the two body interaction potential. We assume:

\begin{enumerate}
	\labitem{\textbf{(H1)}}{equi:nonneg} $f(|p|^2)>0$ for $|p| <\Upsilon$, where $\Upsilon$ is defined as
	$
		\Upsilon = \sup \{ |p|: \,   f(|p|^2) \neq 0  \}.
	$
	\labitem{\textbf{(H2)}}{equi:reg} $f$ is of class $C^{n_0}$ where $n_0 \ge d+3$.
	\labitem{\textbf{(H3)}}{equi:decay} $| \partial_e^{n} f(e)| \lesssim \langle e \rangle^{-n_1 - n}$ for all $e \ge 0$ and $0 \le n \le n_0$, where $n_1 \ge \frac{d+1}{2}$.
	\labitem{\textbf{(H4)}}{assumption:potential} $w$ is a finite positive Borel measure on $\mathbb{R}^d$.
	\labitem{\textbf{(H5)}}{assumption:potential-fourier} $\widehat{w}$ is radial and satisfies $\widehat{w}(k) \ge 0$, $\widehat{w}(0) <\infty$, and $\widehat{w}'(k) \le 0$.
	\labitem{\textbf{(H6)}}{assumption:stability}
	$f$ and $w$ satisfy 
	\begin{equation*}
		1 - \frac{\widehat{w}(0)}{2} \int_{\{ |u| < \Upsilon \}} \frac{\varphi(u)}{(\Upsilon - u)^2} \, du > 0,
	\end{equation*}
	where $\varphi(u) = \int_{\RR^{d-1}} f(u^2 +|v'|^2) \, dv'$ is the marginal of $f$.
\end{enumerate}

Assumptions \ref{equi:nonneg}--\ref{equi:decay} imply that the symbol $f$ is sufficiently smooth and decays rapidly in $e$. Typical examples satisfying \ref{assumption:potential} and \ref{assumption:potential-fourier} include the screened Coulomb potential $\widehat{w}(k) = (1+|k|^2)^{-1}$ and the Dirac delta potential $\widehat{w}(k) =1$.
The properties of marginals are given in Lemma \ref{lem:varphi}. The condition \ref{assumption:stability} serves as a criterion for strong linear stability, see Theorem \ref{thm:Penrose-Lindhard}. Note that this is vacuously true when $\Upsilon = \infty$. We did not pursue optimal assumptions on the regularity and decay of $f$ in order to maintain simplicity.

\subsection{Main result}
Now we are ready to state the main theorem.

\begin{theorem}\label{thm:main}
	Let $f= f(-\Delta)$ be the equilibrium and $w$ be the interaction potential described as in Section \ref{subsec:eq}. Let $d+1 \le N_1 \le  n_0 -2$,  $N_2 \ge d+1$, and $N_3 = \min\{ N_1, N_2 \} - d -1$. Let $\gamma_0$ be the initial density operator satisfying
	\begin{equation}\label{assumption:initial}
		\varepsilon:= \sum_{|\alpha| \le N_1}  \left\| \langle x \rangle^{\lfloor d/2 \rfloor +1} \langle \nabla \rangle^{N_2 + \lfloor d/2 \rfloor +1} \operatorname{ad}_x^{\alpha}(\gamma_0) \langle \nabla \rangle^{N_2 + \lfloor d/2 \rfloor +1} \langle x \rangle^{\lfloor d/2 \rfloor +1}  \right\|_{\mathcal{HS}}  <\infty,
	\end{equation}
	where $\operatorname{ad}_x(A) = [x,A]$ and $\| A \|_{\mathcal{HS}} = \| A \|_{L^2_{x,y}}$.
	Then for sufficiently small $\varepsilon$, there exists a unique global-in-time solution $\gamma(t)$ to \eqref{Hartree-perturbation}. The associated density $\rho_{\gamma}(t)$, defined in \eqref{def:density}, satisfies the phase mixing estimates
	\begin{equation}\label{est:dens-physical}
		\| \partial_x^{n} \rho_{\gamma}(t) \|_{L^{\infty}_{x}} \lesssim \varepsilon \langle t \rangle^{-d- n}
	\end{equation}
	for $0 \le n \le  N_3$. Moreover, there is a unique operator $\gamma_{\infty} \in \mathcal{HS}$ satisfying
	\[
		\| e^{-it\Delta} \gamma(t) e^{it\Delta} - \gamma_{\infty} \|_{\mathcal{HS}} \lesssim \varepsilon \langle t \rangle^{-d/2}.
	\]
\end{theorem}

Theorem \ref{thm:main} establishes the optimal decay rate of the density $\rho(t)$ associated with solution $\gamma(t)$ in the sense that it exhibits the same decay rate as that given by the free Hartree dynamics $\rho^{0}(t) = \rho_{e^{it\Delta} \gamma_0 e^{-it\Delta}}(t)$, with an additional $t^{-1}$ decay for each spatial derivative. For details on the study of free Hartree dynamics, see Section \ref{sec:freeHartree}. Theorem \ref{thm:main} also proves the scattering of $\gamma(t)$ in the Hilbert-Schmidt space. We remark that scattering in Schatten spaces was first established in \cite{Lewin2014} and later extended in \cite{Chen2018, Collot2022, Hadama2025b}.

We will work on the Fourier side to obtain the phase mixing estimates. Namely, we aim to prove
\begin{equation}\label{est:dens-Fourier}
	\left| \widehat{\rho}_{k}(t) \right| \lesssim \varepsilon \langle kt \rangle^{-N_1} \langle k \rangle^{-N_2}.
\end{equation}
The nonlinear iterative scheme involves $L^{\infty}$ based norms in Fourier space, see Section \ref{sec:nonlinear} for details. We remark that the choice of the coordinates $\partial_p^{\alpha} \widehat{\mu}(k-p,p)$ in the bootstrap norm is natural in view of the proof of phase mixing for free Hartree and also helpful for nonlinear estimates, noting that $\partial_p \widehat{V}(k) =0$ for $V = w \star_x \rho$, see Section \ref{subsec:nonlinear-mu} for details.  Moreover, while initial perturbation is assumed to be small and localized in a weighted Hilbert-Schmidt space, this can be relaxed by using a norm on the integral kernel. Namely, the theorem also holds if we replace \eqref{assumption:initial} with 
\begin{equation}\label{assumption:initial2}
	\varepsilon:= \sum_{|\alpha| \le N_1}  \left\| \langle k \rangle^{N_2} \partial_p^{\alpha} \widehat{\gamma_0}(k-p,p) \right\|_{L^{\infty}_k L^1_p} <\infty.
\end{equation}
We proceed with this assumption as such the norm is sufficient for the analysis in this paper.

\subsection{Outline of the paper}
In Section \ref{sec:freeHartree}, we study the decay of the density associated with the free Hartree dynamics. In Section \ref{sec:linear}, we derive the resolvent equation and analyze the linearized operator using Fourier-Laplace analysis. We provide an equivalent condition for spectral stability based on the marginal of the equilibrium. Next, under the spectral stability condition \eqref{thm:Penrose-Lindhard}, we study the Green function associated with $\frac{1}{D(\lambda, k)}$ and establish the pointwise decay in Fourier space. In Section \ref{sec:nonlinear}, we state the bootstrap proposition and prove the nonlinear estimates. In Section \ref{sec:mainthm}, we collect all the results to obtain the phase mixing estimates \eqref{est:dens-physical} and \eqref{est:dens-Fourier} for the density $\rho_{\gamma}(t)$, and prove that the solution to \eqref{Hartree-perturbation} scatters in the Hilbert-Schmidt space.

\subsection{Notations}

We use the Japanese bracket notation $\langle x \rangle = (1+|x|^2)^{1/2}$, and $\langle x,y \rangle = (1+|x|^2+|y|^2)^{1/2}$.
The Laplace transform is given by
\[
	\mathcal{L}[F](\lambda) = \int_{0}^{\infty} e^{-\lambda t} F(t) dt.
\]
The spatial Fourier transform on $\mathbb{R}^d$ is denoted by
\[
	\widehat{g}_{k} = \widehat{g}(k) = \int_{\mathbb{R}^d} e^{-i k \cdot x} g(x) \, dx
\]
The Fourier-Laplace transform is denoted by
\[
	\widetilde{g}(\lambda, k) = \mathcal{L}[\widehat{g}(k)](\lambda) = \int_{0}^{\infty} e^{-\lambda t} \int_{\mathbb{R}^d} e^{-ik \cdot x } g(t, x) \, dx dt.
\]
The commutator of two operators is given by $[A,B] = AB - BA$.
We denote the density function of an operator $A(t)$ as $\rho_{A}(t)$ or $\rho[A](t)$.

\section{Free Hartree dynamics}\label{sec:freeHartree}
In this section, we review the phase mixing estimates for the free Hartree dynamics \begin{equation}
	\label{freeHartree}
	\begin{cases}
		\begin{aligned}
			&i \partial_{t} \gamma^{0} = [ - \Delta , \, \gamma^{0} ] \\
			&\gamma^{0}_{\vert_{t=0}} = \gamma_{0}
		\end{aligned}
	\end{cases}
\end{equation}
given in \cite[Proposition~4.1]{Nguyen2025}. Observe that $\gamma^{0}(t) = e^{it\Delta} \gamma_0 e^{-it\Delta}$. Denote its density by $\rho^{0}(t) = \rho[\gamma^{0}](t)$.
\begin{proposition}\label{prop:freeHartree}
	Let $\gamma_0$ be the initial density operator whose integral kernel $\gamma_0 (x,y)$ satisfies
	\[
		\sum_{|\alpha| \le N_1}  \left\| \langle k \rangle^{N_2} \partial_p^{\alpha} \widehat{\gamma_0}(k-p,p) \right\|_{L^{\infty}_k L^1_p} <\infty
	\]
	for $N_1 , N_2 \ge d+1$.
	Then the density $\rho^{0}(t)$ associated with \eqref{freeHartree} satisfies
	\[
		\left|  \widehat{\rho}^{0}(t,k) \right|  
		\lesssim \langle kt \rangle^{-N_1} \langle k \rangle^{-N_2}
	\]
	for all $k \in \mathbb{R}^d$. It follows that
	\[
		\| \partial_x^{n} \rho^{0}(t) \|_{L^{\infty}_{x}} \lesssim \langle t \rangle^{-d-n}
	\]
	for $0\le n \le \min\{ N_1, N_2 \} - d-1$.
\end{proposition}
\begin{proof}
	Let $\gamma^{0}(t,x,y)$ be the integral kernel of $\gamma^{0}(t)$. Its Fourier transform can be computed as
	\[
		\widehat{\gamma}^{0}(t,k,p) = e^{-it(|k|^2 - |p|^2)} \widehat{\gamma_0}(k,p).
	\]
	Substituting $v= k-2p$, it follows that
	\begin{equation*}
		\begin{split}
			\widehat{\rho}^{0}(t,k) 
			&= \iiint e^{ix \cdot(k'+p-k)} \widehat{\gamma}(t,k',p) \, dx dk' dp 
			= \int \widehat{\gamma}^{0}(t,k-p,p) \, dp \\
			&= \int e^{-itk \cdot (k-2p)} \widehat{\gamma_0}(k-p,p) \, dp.
		\end{split}
	\end{equation*}
	Integrating by parts for $N_1$ times, we obtain that
	\[
		|\widehat{\rho}^{0}(t,k)|
		\lesssim \langle kt \rangle^{-N_1} \langle k \rangle^{-N_2}\sum_{|\alpha| \le N_1} \int \langle k \rangle^{N_2} |\partial_p^{\alpha} \widehat{\gamma_0} (k-p,p) | \, dp.
	\]
	In the physical space, we have
	\[
		\| \partial_x^{n} \rho^{0}(t) \|_{L^{\infty}_x} \lesssim \int_{\mathbb{R}^d} |k|^{n} \langle kt \rangle^{-m_1} \langle k \rangle^{-m_2} \, dk \lesssim \langle t \rangle^{-d -n},
	\]
	for $n \le \min \{ N_1, N_2\} - d-1$.
\end{proof}

\section{Linear estimates}\label{sec:linear}
\subsection{Dispersion relation}
In this section, we write the linearized operator as a Fourier-Laplace multiplier.
\begin{proposition}\label{lem:dispersion}
	For $k \in \mathbb{R}^d \setminus \{ 0 \}$ and $\Re \lambda >0$, we have
	\begin{equation}\label{eq:resolvent}
		\widetilde{\rho}(\lambda, k) = \frac{1}{D(\lambda, k)} \widetilde{S}(\lambda, k)
	\end{equation}
	where
	\begin{equation}\label{eq:dispersion}
		D(\lambda, k)= 1 + i \widehat{w}(k) \int_{\mathbb{R}^d} \frac{f(|p|^2) - f(|k-p|^2) }{\lambda - 2i k \cdot p + i |k|^2} \, dp
	\end{equation}
	and $\widetilde{S}(\lambda,k)$ is the Fourier-Laplace transform of $S(t,x)$, which is given as
	\[
	S(t,x) = \rho\left[ e^{it\Delta} \gamma_{0} e^{-it\Delta} -i \int_{0}^{t} e^{i(t-s)\Delta} [w \star_x \rho_{\gamma}, \gamma ](s) e^{i(s-t)\Delta} \, ds \right](x)
	\]
\end{proposition}

\begin{proof}[Proof of Proposition \ref{lem:dispersion}]
	Let $\gamma(t,x,y)$ be the integral kernel of $\gamma(t)$. Writing \eqref{Hartree-perturbation} in terms of the integral kernel, it follows that
	\begin{equation}\label{eq:kernel}
		i\partial_t \gamma(t,x,y) = (-\Delta_x + \Delta_y) \gamma(t,x,y) + (V(t,x) - V(t,y)) (f(x,y) + \gamma(t,x,y))
	\end{equation}
	where $V= w \star \rho_{\gamma}$. Note that the integral kernel of $f=f(-\Delta)$ can be written as
	\[
	f(x,y) = \int e^{i(x-y) \cdot k'} f(|k'|^2) \, dk' .
	\]
	Its Fourier transform can be computed as
	\begin{equation*}
		\begin{split}
			\widehat{f}(k,p) 
			= \iiint e^{-ik\cdot x -ip \cdot y } e^{i(x-y) \cdot k'} f(|k'|^2) \, dk' dx dy
			= f(|k|^2) \delta_{p=-k}.
		\end{split}
	\end{equation*}
	Taking the spatial Fourier transform to \eqref{eq:kernel}, we get
	\begin{multline*}
		i\partial_t \widehat{\gamma}(t,k,p) = (|k|^2 - |p|^2) \widehat{\gamma}(t,k,p) + \widehat{V}(t,k+p) \left(f(|p|^2) - f(|k|^2)\right) \\
		+ \int \widehat{V}(t,\ell) \left( \widehat{\gamma}(t,k- \ell,p) - \widehat{\gamma}(t,k,p- \ell) \right) d\ell.
	\end{multline*}
	Then take the temporal Laplace transform to obtain
	\begin{multline*}
		(\lambda + i (|k|^2 - |p|^2)) \widetilde{\gamma}(\lambda ,k,p) 
		= \widehat{\gamma_0}(k,p) 
		- i \widetilde{V}(\lambda,k+p) \left( f(|p|^2) - f(|k|^2) \right) \\
		- i \mathcal{L}\left[ \int \widehat{V}(t,\ell) \left( \widehat{\gamma}(t,k-\ell,p) - \widehat{\gamma}(t,k,p-\ell) \right)  d\ell \right](\lambda)
	\end{multline*}
	Noting that
	\begin{equation*}
		\begin{split}
			\widetilde{\rho}(\lambda,k) 
			= \iiint e^{ix \cdot (k' +p -k )} \widetilde{\gamma}(\lambda, k', p) \, dx dk' dp 
			= \int \widetilde{\gamma}(\lambda,k-p,p) \, dp,
		\end{split}
	\end{equation*}
	we write
	\begin{multline*}
		(\lambda + i (|k-p|^2 - |p|^2)) \widetilde{\gamma}(t,k-p,p) 
		= \widehat{\gamma_0}(k-p,p) 
		- i \widehat{w}(k) \widetilde{\rho}(\lambda,k) \left( f(|p|^2) - f(|k-p|^2) \right) \\
		- i \mathcal{L}\left[ \int \widehat{V}(t,\ell) \left( \widehat{\gamma}(t,k-p-\ell,p) - \widehat{\gamma}(t,k-p,p-\ell) \right) d\ell \right](\lambda)
	\end{multline*}
	Dividing both sides by $\lambda +i(|k-p|^2 - |p|^2)$ and integrating with respect to $p$, we get 
	\[
	\widetilde{\rho}(\lambda, k) = \widetilde{S}(\lambda, k) - i \widehat{w}(k) \widetilde{\rho}(\lambda, k) \int \frac{f(|p|^2) - f(|k-p|^2)}{\lambda - 2i k \cdot p + i |k|^2} \, dp.
	\]
	Here $\widetilde{S}(\lambda, k)$ is given by
	\[
	\widetilde{S}(\lambda, k) = \int \frac{\widehat{\gamma_0} (k-p, p)}{\lambda - 2i k \cdot p + i|k|^2} \, dp - i \int \frac{\mathcal{L}[\widehat{F}(k-p, p) ](\lambda)}{\lambda - 2i k \cdot p + i |k|^2} \, dp,
	\]
	where $\widehat{F}$ is defined to be
	\[
	\widehat{F}(t,k,p) = \int \widehat{w}(\ell) \widehat{\rho}(t,\ell) (\widehat{\gamma}(t, k-\ell, p) - \widehat{\gamma}(t, k, p-\ell) ) d\ell.
	\]
	A direct computation shows that
	\[
	\int \frac{\widehat{\gamma_0} (k-p, p)}{\lambda - 2i k \cdot p + i|k|^2} \, dp
	= \int_{0}^{\infty} e^{-\lambda t} \int e^{-it(|k-p|^2 -|p|^2)} \widehat{\gamma_0}(k-p, p) \, dp dt 
	\]
	and
	\begin{equation*}
		\begin{split}
			\int \frac{\mathcal{L}[\widehat{F}(k-p,p)](\lambda)}{\lambda - 2i k \cdot p + i |k|^2} \, dp 
			= \int_{0}^{\infty} e^{-\lambda t} \int_{0}^{t} \! \int e^{-i(t-s)(|k-p|^2-|p|^2)} \widehat{F}(s,k-p, p) \, dp ds dt
		\end{split}
	\end{equation*}
	hold. It follows that
	\[
	\widetilde{S}(\lambda, k) = \int \widetilde{\gamma}^{S} (\lambda, k-p, p) \, dp
	\]
	in which $\gamma^{S}(t)$ is defined by
	\[
	\gamma^{S}(t) =  e^{it\Delta} \gamma_{0} e^{-it\Delta} -i \int_{0}^{t} e^{i(t-s)\Delta} [w \star \rho_{\gamma}, \gamma ](s) e^{i(s-t)\Delta} \, ds.
	\]
	This concludes the proof of \eqref{eq:resolvent} and \eqref{eq:dispersion}.
\end{proof}

\subsection{Spectral stability}
In this section, we prove that the dispersion relation $D(\lambda, k)$ is bounded away from zero on $\{ \Re \lambda  \ge 0 \}$ for certain equilibria.
\begin{theorem}\label{thm:Penrose-Lindhard}
	Let $f= f(-\Delta)$ be the equilibrium satisfying \ref{equi:nonneg} and $\langle \cdot \rangle \widehat{\varphi} \in L^1$. Let $w$ be the interaction potential satisfying \ref{assumption:potential} and \ref{assumption:potential-fourier}. Then
\begin{equation}\label{eq:Penrose-Lindhard}
	\inf_{k \in \mathbb{R}^d \setminus \{ 0\}} \inf_{\Re \lambda \ge 0} | D(\lambda, k)| \ge \theta_{0}
\end{equation}
for some $\theta_{0} >0$ if and only if \ref{assumption:stability} holds.
\end{theorem}
\begin{remark} This implies the invertibility of the linearized operator in $L^2_{t,x}$, which was crucial in \cite{Lewin2015} to obtain the linear stability.
\end{remark}

The Fourier-Laplace analysis for the Coulomb potential case was already done in \cite{Nguyen2025}. Unlike the Coulomb case, our interaction potential is short-range, so $D(\lambda, k)$ is expected to be bounded away from zero for a certain class of equilibrium. The case when $|k| \gtrsim 1$ can be proved in a similar way as in \cite{Nguyen2025}.
However, such an argument cannot be extended to $|k| \ll 1$ because $D(\lambda, k)$ does not have a continuous extension to $\{ \Re \lambda =0 , \ k =0 \}$. Instead, we regard the dispersion relation as a function of $\tilde{\lambda} := \lambda / |k|$ and $k$, which allows us to extend it continuously to $\{ \Re \tilde{\lambda}=0, \  k =0 \}$. The complete proof is given in Appendix \ref{sec:Penrose-Lindhard}. We emphasize that \ref{assumption:stability} plays an important role in spectral stability. Any compactly supported equilibria without \ref{assumption:stability} will have pure imaginary zeros, see Proposition \ref{prop:nozero3}.

\begin{remark}\label{rmk:fermisea}
	For the case of a Fermi gas at zero temperature, i.e., $f(-\Delta) = \chi_{\{-\Delta \le 1 \}}$, the marginal is evaluated as $\varphi(u) = c_d (1 - u^2)^{(d-1)/2} \chi_{\{ |u| \le 1 \}}$ for which $c_d$ is the measure of unit ball in $\mathbb{R}^{d-1}$. Observe that \ref{assumption:stability} does not hold when $d \le 3$. From the argument in Proposition \ref{prop:nozero3}, oscillatory modes exist and the strong stability condition \eqref{eq:Penrose-Lindhard} does not hold for this equilibrium in $d\le 3$. This is compatible with the previous result in \cite{Hadama2023}, as we are in the defocusing case. The asymptotic stability of such equilibria in the defocusing case is beyond the scope of this paper. 
\end{remark}

\subsection{Green function in Fourier space}
In this section, we provide a pointwise estimate of the Green function in Fourier space. We define
\[
	\widetilde{G}(\lambda, k) := \frac{1}{D(\lambda, k)}.
\]
Taking the inverse Laplace transform of \eqref{eq:resolvent}, we get
\[
	\widehat{\rho}_{k}(t) = \int_{0}^{t} \widehat{G}_{k}(t-s) \widehat{S}_{k}(s) \, ds.
\]

\begin{proposition}\label{prop:Green}
	Let $f= f(-\Delta)$ be the equilibrium and $w$ be the interaction potential described as in Section \ref{subsec:eq}. The Green function $\widehat{G}_{k}(t)$ in Fourier space can be written as
	\begin{equation}\label{decomp:Green}
		\widehat{G}_{k}(t) = \delta(t) + \widehat{G}_{k}^{r}(t)
	\end{equation}
	where $\delta(t)$ is the Dirac delta distribution, and the regular part  $\widehat{G}_{k}^{r}(t)$ satisfies
	\begin{equation}
		\label{est:Green-Fourier}
		|\widehat{G}_{k}^{r}(t)  | \le C_{N_0}  |k| \langle kt \rangle^{-\widetilde{N}_{0}},
	\end{equation}
	for all $t \ge 0$, and for some constant $C_{N_0}$ depending only on $N_0$, where $\widetilde{N_0} := N_0 -3$.
\end{proposition}

\begin{proof}
	Recalling \eqref{eq:disp-varphi}, it follows that
	\[
		\widetilde{G}(\lambda, k) = \frac{1}{D(\lambda, k)} = 1 - \frac{\widehat{w}(k) m_{f}(\lambda, k)}{1 + \widehat{w}(k) m_{f}(\lambda, k)}
	\]
	where $m_f (\lambda, k)$ is defined as
	\[
		m_{f}(\lambda, k) = 2 \int_{0}^{\infty} e^{-\lambda t} \sin(t|k|^2) \widehat{\varphi}(2t|k|) \, dt.
	\]
	Note that $|\sin (t|k|^2) | \le \min \{ 1, t|k|^2 \}$ and $ \widehat{w}(k) \lesssim 1$.
	We first estimate $m_{f}(\lambda, k)$. Using the regularity of $\varphi$, we obtain that for $\Re \lambda \ge 0$,
	\[
		| m_{f}(\lambda, k)| \le C_{0} \int_{0}^{\infty} \min \{1, t|k|^2 \} \langle kt \rangle^{-N_{0}} \, dt \le C_0 \langle k \rangle^{-1}
	\]	
	for some absolute constant $C_{0}$. For $\Re \lambda =0$, we need to improve this estimate to get the time decay. Write $\lambda = i \tau$ with $\tau \in \mathbb{R}$. Integrating by parts in $t$ twice, we get
	\begin{align*}
		(|k|^2 \langle k \rangle^2 + \tau^2) m_{f}(i\tau, k )
		&= 2 \int_{0}^{\infty} (|k|^2 \langle k \rangle^2 - \partial_{t}^2) e^{-i \tau t} \sin (t|k|^2) \widehat{\varphi}(2t|k|) \, dt \\
		&= 2 \int_{0}^{\infty} e^{-i\tau t} (|k|^2 \langle k \rangle^2- \partial_{t}^2) (\sin (t|k|^2) \widehat{\varphi}(2t|k|)) \, dt - 2|k|^2 \widehat{\varphi}(0),
	\end{align*}
	noting that $\widehat{\varphi}$ decays rapidly and
	\[
	\partial_{t}^{j} \left( \sin(t|k|^2) \widehat{\varphi}(2t|k|) \right)(0)
	=\begin{cases}
		0, & j=0, \\
		|k|^2 \widehat{\varphi}(0), &  j=1.
	\end{cases}
	\]
	A direct computation shows that
	\begin{equation*}
		\begin{split}
			&(|k|^2 \langle k \rangle^2 - \partial_t^2) (\sin(t|k|^2) \widehat{\varphi}(2t|k|))\\
			&= (2|k|^4 + |k|^2) \sin (t|k|^2) \widehat{\varphi}(2t|k|) - 4|k|^3 \cos (t|k|^2) \widehat{\varphi}'(2t|k|) - 4|k|^2 \sin (t|k|^2) \widehat{\varphi}''(2t|k|) \\
			&\lesssim (|k|^2 \langle k \rangle^2 |\sin(t|k|^2)| + |k|^3) \langle kt \rangle^{-N_0}. 
		\end{split}
	\end{equation*}
	Hence
	\begin{align*}
		|m_{f}(i\tau,  k)| 
		\le C_{0} \int_{0}^{\infty}
			\frac{|k|^2\langle k \rangle^{2} |\sin(t|k|^2)| + |k|^3}{|k|^2 \langle k \rangle^2 + \tau^2} \langle kt \rangle^{-N_{0}} \, dt + \frac{C_{0} |k|^2 \widehat{\varphi}(0)}{|k|^2 \langle k \rangle^2 + \tau^2}
		\le \frac{C_{0}|k|^2 \langle k \rangle}{|k|^2 \langle k \rangle^2 + \tau^2}.
	\end{align*}
	Noting that $\partial_{\lambda}^{n} e^{-\lambda t} = (-t)^{n} e^{-\lambda t}$, we obtain the estimates for $\partial_{\lambda}^{n} m_{f}(\lambda, k)$. Namely,
	\[
		|\partial_{\lambda}^{n} m_{f}(\lambda, k)| \le C_{N_{0}} \int_{0}^{\infty} t^n \langle kt \rangle^{-N_{0}} \, dt \le C_{N_{0}} |k|^{-n} \langle k \rangle^{-1}
	\]	
	for $0\le n < N_{0} -1$ and $\Re \lambda \ge 0$. Moreover,
	\begin{align*}
		|\partial_{\lambda}^{n} m_{f}(\lambda, k)| 
		\le C_{N_{0}} \int_{0}^{\infty}
		\frac{|k|^2\langle k \rangle^{2} |\sin(t|k|^2)| + |k|^3}{|k|^2 \langle k \rangle^2 + |\Im \lambda|^2} t^n \langle kt \rangle^{-N_{0}} \, dt 
		\le \frac{C_{N_{0}}|k|^{2-n} \langle k \rangle}{|k|^2 \langle k \rangle^2 + |\Im \lambda|^2}.
	\end{align*}
	for $0 \le n < N_{0} -2$ and $\Re \lambda =0$.
	Using the stability condition \eqref{eq:Penrose-Lindhard} and assumptions \ref{assumption:potential} and \ref{assumption:potential-fourier} on $\widehat{w}$, we get
	\[
		|\partial_{\lambda}^{n} \widetilde{G}_{k}^{r} (\lambda)| 
		\le 
		\left| \sum_{j=0}^{n} \binom{n}{j} \partial_{\lambda}^{j} (1-D(\lambda, k)) \partial_{\lambda}^{n-j} (D(\lambda, k)^{-1}) \right| 
		\le \frac{C_{N_{0}} |k|^{2-n} \langle k \rangle}{|k|^2 \langle k \rangle^2 + |\Im \lambda|^2}
	\]
	for $0 \le n < N_{0} -2$ and $\Re \lambda \ge 0$. Taking the inverse Laplace transform and integrating by parts,
	\[
		\widehat{G}_{k}^{r}(t)  =  \frac{1}{2\pi} \int_{\mathbb{R}} e^{i\tau t} \widetilde{G}_{k}^{r}(i \tau) \, d\tau
		= \frac{(-1)^{n}}{2\pi t^n} \int_{\mathbb{R}} e^{i\tau t} \partial_{\lambda}^{n} \widetilde{G}_{k}^{r}(i\tau) \, d\tau
	\]
	for $0 \le n < N_{0}-2$. Take $n= N_{0} -3$. Since $\int_{\mathbb{R}} (a^2 + x^2)^{-1} dx = \pi |a|^{-1}$, we get
	\[
		|\widehat{G}_{k}^{r}(t)| \le \frac{C_{N_{0}}}{t^{N_{0}- 3}} \int_{\mathbb{R}} \frac{|k|^{-N_{0}+5} \langle k \rangle}{|k|^2 \langle k \rangle^2 + \tau^2} \, d \tau
		\le C_{N_{0}} |k| |kt|^{-N_{0} +3},
	\]
	which concludes the proof of \eqref{est:Green-Fourier}.
\end{proof}

\section{Nonlinear estimates}\label{sec:nonlinear}
Recall that $N_1$, $N_2$, and $N_3$ are chosen so that \[
		d+1 \le N_1 \le  n_0 -2,  \quad  N_2 \ge d+1, \quad N_3 = \min\{ N_1, N_2 \} - d -1.
\]  Let $\mu(t) = e^{-it\Delta} \gamma(t) e^{it\Delta}$ be the conjugate of $\gamma(t)$. Define the norms as
\begin{equation*}
	\begin{split}
		\| \mu \|_{X_{T}^{N}}
			&= \sup_{t \in[0, T]} \sum_{|\alpha| \le N} \left\| \langle k  \rangle^{N_2} \partial_p^{\alpha} \widehat{\mu}(t,k-p,p) \right\|_{L^{\infty}_k L^1_p}, \\
		\| \rho \|_{Y_{T}} 
			&= \sup_{t \in[0, T]} \left\| \langle kt \rangle^{N_1} \langle k \rangle^{N_2} \widehat{\rho}_{k}(t) \right\|_{L^{\infty}_{k}}, \\
		\| \mu \|_{Z_{T}} &= \sup_{t \in[0,T]}  \left( \|\mu \|_{X_{t}^{N_1 -2}} +  \langle t \rangle^{-\delta}\| \mu \|_{X_{t}^{N_1 -1}} + \langle t \rangle^{-1} \| \mu \|_{X_{t}^{N_1}}\right),
	\end{split}
\end{equation*}
where $\delta >0$ is a fixed small number.
Now we state and prove the bootstrap proposition.
\begin{proposition}\label{prop:bootstrap}
	Let $f= f(-\Delta)$ be the equilibrium and $w$ be the interaction potential described as in Section \ref{subsec:eq}. Assume that the initial density operator $\gamma_0$ satisfies \eqref{assumption:initial}. Let $\gamma(t)$ be the solution of \eqref{Hartree-perturbation} with initial density operator $\gamma_0$ for $0 \le t \le T$. Then the following a priori estimates
	\begin{align}
			\| \mu\|_{Z_{t}} &\lesssim \varepsilon + \| \rho \|_{Y_{t}} + \| \rho \|_{Y_{t}} \| \mu \|_{Z_{t}}, \label{est:nonlinear-mu} \\
			\| S \|_{Y_{t}} &\lesssim \varepsilon + \| \rho \|_{Y_{t}} \| \mu \|_{Z_{t}}, \label{est:nonlinear-source}  \\
			\| \rho \|_{Y_{t}} &\lesssim \| S \|_{Y_{t}}, \label{est:linear-density} 
	\end{align}
	hold for all $0\le t \le T$.
\end{proposition}

\subsection{Nonlinear estimates on the profile}\label{subsec:nonlinear-mu}

In this section, we prove \eqref{est:nonlinear-mu}. Recall that
\begin{equation}\label{eq:mu}
\begin{split}
	\widehat{\mu}(t,k,p)
	&= \widehat{\gamma_0}(k,p) - i \int_{0}^{t} e^{is(|k|^2 -|p|^2) } \widehat{w}(k+p) \widehat{\rho}(s, k+p) \left( f(|p|^2) - f(|k|^2) \right) ds \\
	&\quad - i \int_{0}^{t} \int_{\mathbb{R}^d} \widehat{w}(\ell) \widehat{\rho}(s,\ell) \left( e^{i \ell s \cdot (2k-\ell)} \widehat{\mu}(s,k-\ell,p) - e^{-i\ell s \cdot (2p-\ell)} \widehat{\mu}(s,k, p-\ell) \right) d\ell ds \\
	&=: \mathcal{I}(k,p) + \mathcal{R}^{L}(t,k,p) + \mathcal{R}^{NL}(t,k,p)
\end{split}
\end{equation}
We first estimate the initial data term $\mathcal{I}(k,p) = \widehat{\gamma_0}(k,p)$. Observe that
\[
	\| \mathcal{I} \|_{X_{T}^{N_1}} 
	\lesssim \sum_{|\alpha| \le N_1} \| \langle k \rangle^{N_2} \partial_p^{\alpha} \widehat{\gamma_0}(k-p,p) \|_{L^{\infty}_k L^1_p}
	\lesssim \varepsilon,
\]
which follows from \eqref{assumption:initial}. 

Next, we estimate the linear term. Let $m \le N_1$. Observe that 
	\begin{equation*}
	\begin{split}
		\langle k\rangle^{N_2} \sum_{|\alpha| \le m} \int |\partial_p^{\alpha} \mathcal{R}^{L}(t,k-p,p)| \, dp
		\lesssim \langle k \rangle^{N_2} \sum_{\substack{|\alpha_1| +|\alpha_2| \le m}} \int|\mathcal{R}^{L}_{\alpha_1, \alpha_2}(t,k-p,p)|\, dp
	\end{split}
	\end{equation*}
where
\begin{equation*}
	\begin{split}
		\mathcal{R}^{L}_{\alpha_1, \alpha_2}(t,k-p,p)
		&=  \int_{0}^{t} \partial_p^{\alpha_1 } e^{isk \cdot(k-2p)} \, 
		 \widehat{V}(s,k) \partial_p^{\alpha_2} (f(|p|^2) - f(|k-p|^2)) \, ds,
	\end{split}
\end{equation*}
noting that $V = w \star_x \rho$ and $\partial_p \widehat{V}(k) =0$.
Observe that
\begin{equation}\label{eq:profile-linear}
	\begin{split}
		\int |\mathcal{R}^{L}_{\alpha_1, \alpha_2}(t,k-p,p)| \, dp
		&\lesssim \| \rho \|_{Y_{t}} 
			\int_{0}^{t}  \langle k \rangle^{-N_2} \langle ks \rangle^{-N_1 +|\alpha_1|}
			\int \left| \partial_p^{\alpha_2} (f(|p|^2) - f(|k-p|^2)) \right| dp \, ds 
	\end{split}
\end{equation}
Consider the case when $|k| \ll 1$. Let $g(p)= f(|p|^2)$ for notational convenience. Substituting $k' = k+p$ and $p' = k-p$, it follows that
\begin{equation*}
	\begin{split}
		\partial_p^{\alpha_2} (f(|p|^2) - f(|k-p|^2))
		= \partial_{p}^{\alpha_2}\left( g(p) - g(p-k) \right) 
		\lesssim \langle p \rangle^{-2n_1} |k|
	\end{split}
\end{equation*}
for $|k| \ll 1$, upon using the Taylor's theorem on $\nabla^{\alpha_2} g$, provided that $N_1 + 1 \le n_0$. 
Then \eqref{eq:profile-linear} becomes
\begin{equation*}
	\begin{split}
		\langle k \rangle^{N_2} \int \left|\mathcal{R}^{L}_{\alpha_1, \alpha_2}(t,k-p,p) \right| \, dp
		&\lesssim \| \rho \|_{Y_{t}} 
		|k|\int_{0}^{t} \langle ks \rangle^{-N_1 +|\alpha_1|} \,  ds,
	\end{split}
\end{equation*}
where we used $2n_1 > d$.
Note that 
\begin{equation}\label{eq:s(k+p)}
	\begin{split}
			\int_{0}^{t} 
			\langle ks \rangle^{-N_1 +|\alpha_1|} \, ds
		\lesssim 
		\begin{cases}
			|k|^{-1},					&\textrm{if} \quad |\alpha_1| \le N_1 -2 \\[5pt]
			t^{\delta} \, |k|^{-1+\delta} 	, \quad &\textrm{if} \quad |\alpha_1|= N_1 -1 \\[5pt]
			t 	, 				&\textrm{if} \quad |\alpha_1| = N_1.
		\end{cases}
	\end{split}
\end{equation}
It remains to consider the case when $|k| \gtrsim 1$. From the decay of $f$, we obtain that
\begin{equation*}
	\begin{split}
		\partial_p^{\alpha_2} (f(|p|^2) - f(|k-p|^2)) \lesssim \langle p \rangle^{-2n_1} + \langle k-p \rangle^{-2n_1}.
	\end{split}
\end{equation*}
which is also integrable in $p$.
It follows from \eqref{eq:profile-linear} that
\begin{equation*}
	\begin{split}
		\langle k \rangle^{N_2} \int \left|\mathcal{R}^{L}_{\alpha_1, \alpha_2}(t,k-p,p) \right| \, dp
		&\lesssim 
		\| \rho \|_{Y_{t}}
		\int_{0}^{t} \langle ks \rangle^{-N_1 + |\alpha_1|}\, ds.
	\end{split}
\end{equation*}
Using \eqref{eq:s(k+p)}, we conclude that
\begin{equation}
	\begin{split}
		\| \mathcal{R}^{L} \|_{X_{t}^{N_1 -2 }}  
		+ \langle t\rangle^{-\delta} \|  \mathcal{R}^{L} \|_{X_{t}^{N_1 -1}} 
		+ \langle t\rangle^{-1} \|  \mathcal{R}^{L} \|_{X_{t}^{N_1}}
		&\lesssim \| \rho \|_{Y_{t}}
	\end{split}
\end{equation}
for $0\le t \le T$.

Now we estimate the nonlinear term given by
\[
	\mathcal{R}^{NL}(t,k,p) 
	= - i \int_{0}^{t} \int_{\mathbb{R}^d} \widehat{w}(\ell) \widehat{\rho}(s,\ell) \left( e^{i \ell s \cdot (2k-\ell)} \widehat{\mu}(s,k-\ell,p) - e^{-i\ell s \cdot (2p-\ell)} \widehat{\mu}(s,k, p-\ell) \right) d\ell ds.
\]
We focus on the first term as the other could be estimated in the same way. Define
\[
	\mathcal{R}^{NL, 1} (t,k,p) = \int_{0}^{t} \int_{\mathbb{R}^d} e^{i \ell s \cdot (2k- \ell)} \widehat{w}(\ell)   \widehat{\rho}(s,\ell)  \widehat{\mu}(s,k- \ell, p) \, d\ell ds.
\]
Using $\widehat{w}\in L^{\infty}$ and
\[
	\frac{\langle k \rangle^{N_2}}{\langle k- \ell \rangle^{N_2} \langle \ell \rangle^{N_2}} \lesssim \frac{1}{\min (\langle k- \ell \rangle, \langle \ell \rangle)^{N_2}},
\]
we obtain that
\begin{equation*}
	\begin{split}
		&\langle k\rangle^{N_2} \sum_{|\alpha| \le m} \int_{\RR^d} |\partial_p^{\alpha}  \mathcal{R}^{NL, 1}(t,k-p,p)| \, dp \\
		&\lesssim \langle k \rangle^{N_2}  
		\int_{0}^{t} \int_{\mathbb{R}^d} |\widehat{\rho}(s,\ell)| 
		\sum_{\substack{\alpha_1 + \alpha_2 \le m} } |\ell s|^{|\alpha_1|} \int_{\RR^d} \left| \partial_p^{\alpha_2} \widehat{\mu}(s,k-p- \ell,p)\right| dp \, d\ell \,  ds \\
		&\lesssim \| \rho \|_{Y_{t}} \sum_{n \le m }
		\int_{0}^{t}   \| \mu\|_{X_{s}^{ m- n}} \int_{\mathbb{R}^d}  \frac{\langle \ell s \rangle^{-N_1} }{\min (\langle k- \ell \rangle, \langle \ell \rangle)^{N_2}}
			 | \ell s|^{n}   d\ell \,  ds
	\end{split}
\end{equation*}
for any $m$. When $m \le N_1 -2$, we bound $\langle \ell s \rangle^{-N_1} |\ell s|^{n} \lesssim \langle  \ell s \rangle^{-2}$ to get 
\[
	\| \mathcal{R}^{NL, 1} \|_{X_{t}^{N_1 -2 }} \lesssim \| \rho \|_{Y_{t}} \| \mu \|_{Z_{t}} \int_{0}^{t} \langle s \rangle^{-2} \, ds \lesssim \| \rho \|_{Y_{t}} \| \mu \|_{Z_{t}}.
\]
When $m= N_1 - 1$, we consider the cases $n= 0$ and $1 \le n \le N_1 -1$ separately to obtain
\begin{equation*}
\begin{split}
	 \|  \mathcal{R}^{NL, 1} \|_{X_{t}^{N_1 -1}}
	 &\lesssim \| \rho \|_{Y_{t}} \sum_{n \le N_1 -1} \int_{0}^{t} \| \mu\|_{X^{N_1 - 1 - n}_{s}} 
	 	\int_{\mathbb{R}^d}  \frac{ \langle \ell s \rangle^{-N_1} }{\min (\langle k- \ell \rangle, \langle \ell \rangle)^{N_2}}
	 	| \ell s|^{n}   d\ell ds \\
	 &\lesssim \| \rho \|_{Y_{t}} \| \mu \|_{Z_{t}} \int_{0}^{t} (\langle s \rangle^{-d + \delta} + \langle s \rangle^{-1}) \, ds \\
	 &\lesssim t^{\delta} \, \| \rho \|_{Y_{t}} \| \mu \|_{Z_{t}}.
\end{split}
\end{equation*}
When $m = N_1$, we estimate it in a similar way to get
\begin{equation*}
	\begin{split}
		\|  \mathcal{R}^{NL, 1} \|_{X_{t}^{N_1}} 
		&\lesssim \| \rho \|_{Y_{t}} \sum_{n \le N_1} \int_{0}^{t} \| \mu\|_{X^{N_1  - n}_{s}} 
		\int_{\mathbb{R}^d}  \frac{\langle \ell s \rangle^{-N_1} }{\min (\langle k- \ell \rangle, \langle \ell \rangle)^{N_2}}
		| \ell s|^{n}   d\ell ds \\
		&\lesssim \| \rho \|_{Y_{t}} \| \mu \|_{Z_{t}} \int_{0}^{t} (\langle s \rangle^{-d + 1} +  \langle s \rangle^{-d + \delta} + 1) \, ds \\
		&\lesssim t \, \| \rho \|_{Y_{t}} \| \mu \|_{Z_{t}}.
	\end{split}
\end{equation*}
which completes the proof of
\begin{equation}
	\begin{split}
		\| \mathcal{R}^{NL} \|_{X_{t}^{N_1 -2}}  
		+ \langle t \rangle^{-\delta}\|  \mathcal{R}^{NL} \|_{X_{t}^{N_1 -1}}
		+ \langle t \rangle^{-1}  \|  \mathcal{R}^{NL} \|_{X_{t}^{N_1}}
		&\lesssim \| \rho \|_{Y_{t}} \| \mu \|_{Z_{t}}
	\end{split}
\end{equation}
for any $0\le t \le T$. To sum up, we obtain \eqref{est:nonlinear-mu}.

\subsection{Nonlinear estimates on the source term}\label{subsec:nonlinear-source}
In this section, we estimate the source term and prove \eqref{est:nonlinear-source}. We write
\begin{equation*}
\begin{split}
	\widehat{S}_{k}(t)
	&= \int e^{-itk\cdot(k-2p)} \widehat{\gamma_0}(k-p,p) \, dp \\
	& \quad - i\int_{0}^{t} \int \widehat{w}(\ell) \widehat{\rho}(s,\ell) e^{-it|k|^2} \int e^{2i(kt- \ell s) \cdot p} \left( e^{i \ell s\cdot(2k- \ell)} \widehat{\mu}(s,k-p- \ell ,p) - e^{i|\ell|^2 s} \widehat{\mu}(s, k-p, p-\ell) \right) dp \, d\ell \, ds \\
	&= \widehat{\rho}^{0}_{k}(t) + \widehat{\rho}^{NL}_{k} (t)
\end{split}
\end{equation*}
The first term is exactly the density associated with the free Hartree dynamics. Under the smallness assumption \eqref{assumption:initial} on the initial data, it follows that
\begin{equation}
	\langle kt \rangle^{N_1} \langle k \rangle^{N_2} \left| \widehat{\rho}^{0}_{k}(t) \right| \lesssim \varepsilon,
\end{equation}
for $0 \le t \le T$, see Proposition \ref{prop:freeHartree} for details.

Next, we estimate the nonlinear term. Recall that
\begin{equation*}
\begin{split}
	& \widehat{\rho}^{NL}_{k}(t)  \\
	&= - i\int_{0}^{t} \int \widehat{w}(\ell) \widehat{\rho}(s,\ell) e^{-it|k|^2} \int e^{2i(kt- \ell s) \cdot p} \left( e^{i \ell s\cdot(2k- \ell)} \widehat{\mu}(s,k-p- \ell ,p) - e^{i|\ell|^2 s} \widehat{\mu}(s, k-p, p-\ell) \right) dp \, d\ell \, ds.
\end{split}
\end{equation*}
We focus on the first term as the other one can be estimated in a similar way. Define
\[
	\widehat{\rho}^{NL,1}_{k}(t) =  e^{-it|k|^2}  \int_{0}^{t} \int \widehat{w}(\ell) \widehat{\rho}(s,\ell) e^{i \ell s\cdot(2k- \ell)} \int e^{2i(kt- \ell s) \cdot p}   \widehat{\mu}(s,k-p- \ell ,p) \, dp \, d\ell \, ds.
\]
Integrating by parts with respect to $p$ yields
\begin{equation*}
\begin{split}
	\langle kt \rangle^{N_1} \langle k \rangle^{N_2}| \widehat{\rho}^{NL,1}_{k}(t)| 
	&\lesssim \| \rho \|_{Y_{t}} \int_{0}^{t} \| \mu \|_{X_{s}^{N_1}}
	\int \frac{\langle kt \rangle^{N_1}}{\langle \ell s \rangle^{N_1} \langle kt - \ell s \rangle^{N_1}} 
	\frac{\langle k \rangle^{N_2}}{\langle \ell \rangle^{N_2}\langle k-\ell \rangle^{N_2}} 
	 \,   d\ell  \, ds \\
	&\lesssim \| \rho \|_{Y_{t}} \| \mu \|_{Z_t} \int_{0}^{t} \langle s \rangle^{-d + 1} \, ds \\
	&\lesssim \| \rho\|_{Y_{t}} \| \mu \|_{Z_t}
\end{split}
\end{equation*}
given that $d\ge 3$. It follows that
	\begin{equation}
		\langle kt \rangle^{N_1} \langle k \rangle^{N_2} \left| \widehat{\rho}^{NL}_{k}(t) \right| \lesssim \| \rho \|_{Y_{t}} \| \mu \|_{Z_{t}}
	\end{equation}
for $0\le t\le T$, which concludes the proof of \eqref{est:nonlinear-source}.

\subsection{Proof of Proposition \ref{prop:bootstrap}}\label{sec:bootstrap}
Note that \eqref{est:nonlinear-mu} and \eqref{est:nonlinear-source} are proved in Sections \ref{subsec:nonlinear-mu} and \ref{subsec:nonlinear-source}, respectively. It remains to prove \eqref{est:linear-density}.
It follows from \eqref{eq:resolvent} and \eqref{decomp:Green} that
\[
	\widehat{\rho}_{k}(t) = \widehat{S}_{k}(t) + \int_{0}^{t} \widehat{G}^{r}_{k}(t-s) \widehat{S}_{k}(s) \, ds.
\]
Multiplying both sides by $\langle kt \rangle^{N_1} \langle k \rangle^{N_2}$, we have
	\begin{equation*}
		\begin{split}
			\langle kt \rangle^{N_1} \langle k \rangle^{N_2} | \widehat{\rho}_{k}(t) |
			&\lesssim \langle kt \rangle^{N_1} \langle k \rangle^{N_2} 
			|\widehat{S}_{k}(t)| + \langle kt \rangle^{N_1} \langle k \rangle^{N_2} \left|  \int_{0}^{t} \widehat{G}^{r}_{k}(t-s) \widehat{S}_{k}(s)\, ds. \right|
		\end{split}
	\end{equation*}
	The first term on the right hand side is bounded by the $Y_t$ norm. Namely,
	\[
	\langle kt \rangle^{N_1} \langle k \rangle^{N_2 } |\widehat{S}_{k}(t) | \lesssim \| S \|_{Y_{t}}
	\]
	for $0 \le t \le T$.
	On the other hand, we use \eqref{est:Green-Fourier} to estimate the other term as
	\begin{equation*}
		\begin{split}
			 \langle kt \rangle^{N_1} \langle k \rangle^{N_2} \left|  \int_{0}^{t} \widehat{G}^{r}_{k}(t-s) \widehat{S}_{k}(s)\, ds \right|
			&\lesssim  \| S \|_{Y_{t}}   |k| \int_{0}^{t} \frac{  \langle kt \rangle^{N_1}}{\langle k(t-s) \rangle^{\widetilde{N}_{0}} \langle ks \rangle^{N_1}} ds \\
			&\lesssim \| S \|_{Y_{t}} |k|  \left( \int_{0}^{t/2}  \langle ks \rangle^{-N_1} ds + \int_{t/2}^{t} \langle k(t-s) \rangle^{-N_1} \, ds \right) \\
			&\lesssim \| S \|_{Y_{t}}
		\end{split}
	\end{equation*}
	noting that $N_1 \le \widetilde{N}_0 = N_0 -3$. This completes the proof of \eqref{est:linear-density}.

\subsection{Proof of Theorem \ref{thm:main}}\label{sec:mainthm}
In this section, we give a proof of Theorem \ref{thm:main}. Define $\zeta(t)$ by
\[
\zeta(t):= \| S \|_{Y_{t}} + \| \mu \|_{Z_{t}}.
\]
From \eqref{est:nonlinear-source}, we have
\[
	\| S \|_{Y_{t}} \lesssim \varepsilon + \| S \|_{Y_{t}} \| \mu \|_{Z_{t}}.
\]
On the other hand, it follows from \eqref{est:nonlinear-mu} that
\[
	\| \mu \|_{Z_{t}} \lesssim  \varepsilon + \| S \|_{Y_{t}} + \| S \|_{Y_{t}} \| \mu \|_{Z_{t}} \lesssim \varepsilon + \| S \|_{Y_{t}} \| \mu \|_{Z_{t}}.
\]
Combining these two, we have
\[
	\zeta(t) \le C_0 \varepsilon + C_1 \zeta(t)^2.
\]
for some absolute constants $C_0$ and $C_1$. From the local well-posedness theory and the standard continuation argument, $\zeta(t)$ exists for all $t\ge0$ and satisfies $\zeta(t) \le 2 C_0 \varepsilon$ for sufficiently small $\varepsilon >0$. This proves that $$\| \rho \|_{Y_{t}} \lesssim \varepsilon$$ for all $t\ge 0$.
This implies that
\[
\| \partial_x^{n} \rho(t) \|_{L^{\infty}_x} \lesssim \varepsilon\int_{\mathbb{R}^d} |k|^{n} \langle kt \rangle^{-N_1} \langle k \rangle^{-N_2} \, dk \lesssim \varepsilon \langle t \rangle^{-d -n}
\]
for $0 \le n \le N_3$.

Finally, we establish scattering of $\gamma(t)$ in the Hilbert-Schmidt space. Recall from \eqref{eq:mu} that
\begin{equation*}
	\begin{split}
		\frac{1}{2} \frac{d}{dt} \| \widehat{\mu}(t) \|_{L^2_{k,p}}^2
		&= \Im \iint e^{it(|k|^2 - |p|^2)} \widehat{V}(t, k+p) (f(|p|^2) - f(|k|^2)) \overline{\widehat{\mu}}(t,k,p) \, dk \, dp\\
		&\quad + \Im \iiint \widehat{V}(t, \ell) \left( e^{i \ell t \cdot(2k-\ell)} \widehat{\mu}(t,k- \ell,p) - e^{-i \ell t \cdot(2p- \ell)} \widehat{\mu}(t,k,p-\ell) \right) \! \overline{\widehat{\mu}}(t,k,p) \, dk \, dp  \, d\ell
	\end{split}
\end{equation*}
We first prove
\begin{equation}\label{est:scattering-g}
	\int_{\mathbb{R}^d} \left| g(p-k) - g (p+k) \right|^2 dp \lesssim |k|^2
\end{equation}
where $g(p)= f(|p|^2)$. Note that $\| g \|_{H^1} < \infty$. We use the mean value theorem to get
\begin{equation*}
	\begin{split}
		\left|g(p-k) - g(p+k)\right|^2
		\lesssim \left| \int_{-1}^{1} \nabla g \left( p+\theta k\right) \cdot k \, d\theta \right|^2 
		\lesssim |k|^2 \int_{-1}^{1} \left| \nabla g \left( p+\theta k\right) \right|^2  d \theta.
	\end{split}
\end{equation*}
Integrating with respect to $p$, it follows that
\begin{equation*}
	\begin{split}
		\int_{\mathbb{R}^d} \left| g(p-k) - g(p+k) \right|^2 dp
		&\lesssim |k|^2 \int_{-1}^{1} \int_{\mathbb{R}^d} \left| \nabla g (p+\theta k)\right|^2 dp \, d\theta \lesssim |k|^2 \| \nabla g \|_{L^2}^2 \lesssim |k|^2,
	\end{split}
\end{equation*}
which concludes the proof of \eqref{est:scattering-g}. 
Now we estimate the linear term.
Substituting $k' = \frac{k+p}{2}$ and $p' = \frac{k-p}{2}$, we obtain that
\begin{equation*}
\begin{split}
	\iint &\left| e^{it(|k|^2-|p|^2)} \widehat{V}(t,k+p) (f(|p|^2) - f(|k|^2)) \overline{\widehat{\mu}}(t,k,p) \right| dk \, dp \\
	&\lesssim \| \widehat{\mu}(t) \|_{L^2_{k,p}}
		\left(\iint \left| e^{isk' \cdot p'} \widehat{V}(t,k') \left( g(p' -k') - g(p'+k') \right) \right|^2 dk'  dp' \right)^{\! 1/2}  \\
	&\lesssim \| \widehat{\mu}(t) \|_{L^2_{k,p}}
		\left(\int |\widehat{V}(t,k')|^2 \int \left| g(p' -k') - g(p'+k') \right|^2 dp'  dk' \right)^{\! 1/2} \\
	&\lesssim \varepsilon \| \widehat{\mu}(t) \|_{L^2_{k,p}}
		\left(\int |k|^2 \langle k \rangle^{-2N_2}  \langle kt \rangle^{-2N_1}  \, dk \right)^{\! 1/2} \\
	&\lesssim \varepsilon \langle t \rangle^{-1- d/2} \| \widehat{\mu}(t) \|_{L^2_{k,p}} .
\end{split}
\end{equation*}
On the other hand, one of the nonlinear term can be estimated as
\begin{equation*}
	\begin{split}
		\iiint \left| e^{i\ell t \cdot (2k - \ell)} \widehat{V}(t, \ell) \widehat{\mu}(t,k-\ell,p) \overline{\widehat{\mu}}(t,k,p) \right|  dp  dk  d\ell
		&\lesssim \| \widehat{V}(t) \|_{L^1_{k}} \| \widehat{\mu}(t)\|_{L^2_{k,p}}^2 \lesssim \varepsilon \langle t \rangle^{-d} \| \widehat{\mu}(t) \|_{L^2_{k,p}}^2,
	\end{split}
\end{equation*}
where the last inequality follows from
\[
\| \widehat{V}(t) \|_{L^1_k}
\lesssim \varepsilon \int_{\mathbb{R}^d} \langle k \rangle^{-N_2} \langle kt \rangle^{-N_1} \, dk
\lesssim \varepsilon \langle t \rangle^{-d}.
\]
The other term can be estimated in a similar way. Hence we obtain
\[
	\frac{d}{dt} \| \widehat{\mu}(t) \|_{L^2_{k,p}} \lesssim \varepsilon \langle t \rangle^{-1-d/2} + \varepsilon \| \widehat{\mu} (t) \|_{L^2_{k,p}} \langle t \rangle^{-d}.
\]
Applying Gronwall's lemma, we obtain that $\| \widehat{\mu}(t) \|_{L^2_{k,p}}$ is uniformly bounded in $t$ so that
\[
	\frac{d}{dt} \| \widehat{\mu}(t)\|_{L^2_{k,p}} \lesssim \varepsilon \langle t \rangle^{-1-d/2}.
\]
It follows that 
\[
	\| \mu(t_1) - \mu(t_2) \|_{L^2_{x,y}} \lesssim \varepsilon \int_{t_2}^{t_1}  \langle s \rangle^{-1-d/2} \, ds \lesssim \varepsilon \langle t_2 \rangle^{-d/2}
\]
for all $t_1 \ge t_2 \ge 0$. Hence $\mu(t)$ has a unique limit $\gamma_{\infty} \in \mathcal{HS}$ as $t \rightarrow \infty$. Moreover, there holds
\[
	\| e^{-it\Delta} \gamma(t) e^{it\Delta} - \gamma_{\infty} \|_{\mathcal{HS}}
	=\| \mu(t) - \gamma_{\infty} \|_{L^2_{x,y}} \lesssim \varepsilon \langle t \rangle^{-d/2},
\]
which completes the proof of Theorem \ref{thm:main}.

\appendix
\section{Proof of Theorem \ref{thm:Penrose-Lindhard}}\label{sec:Penrose-Lindhard}
In this section, we prove Theorem \ref{thm:Penrose-Lindhard}. We begin with the properties of marginal.

\begin{lemma}\label{lem:varphi}
	Let $f=f(-\Delta)$ be the equilibrium satisfying \ref{equi:nonneg}. Let 
	\begin{equation}\label{def-marginals}
		\varphi(u) =\int_{\mathbb{R}^{d-1}} f (u^2 + |v'|^2) \, dv'
	\end{equation}
	be the marginal of $f$.
	Then, $\varphi$ is an even function on $\mathbb{R}$ and $\varphi'(u) <0$ for $0<u<\Upsilon$.
	If we further assume \ref{equi:reg} and \ref{equi:decay}, then there also hold
	\begin{itemize}
		\item $\varphi$ is of class $C^{N_0}$ where $N_0 = n_{0} + \lfloor \frac{d-1}{2} \rfloor$.
		\item $\varphi$ and its derivatives satisfy  
		\begin{equation}\label{decay-varphi}
			|\partial_{u}^{n} \varphi(u) | \lesssim \langle u \rangle^{-N_1-n}
		\end{equation} 
		for all $u \in \mathbb{R}$ and $0 \le n \le N_{0}$ where $N = 2n_1 - d + 1$.
		\item $\widehat{\varphi}$ and its derivatives satisfy 
		\begin{equation}
			|\partial_t^{n} \widehat{\varphi}(t) | \lesssim \langle t \rangle^{-N_0}
		\end{equation}
		for $0 \le n \le N$.
	\end{itemize}
\end{lemma}

\begin{proof}
	Refer to \cite[Lemma~2.2]{Nguyen2025}.
\end{proof}

\begin{lemma}
	For $k \in \mathbb{R}^d \setminus \{0\}$ and $\Re \lambda >0$, we can write the dispersion relation as
	\begin{equation}\label{def:Dlambda1}
		D(\lambda, k)= 1 + \frac{\widehat{w}(k)}{2|k|} \left[ \mathcal{H}\left( \frac{-i\lambda + |k|^2}{2|k|} \right) - \mathcal{H} \left( \frac{-i \lambda - |k|^2}{2|k|} \right) \right],
	\end{equation}
	where
	\[
	\mathcal{H}(z) = \int_{\mathbb{R}} \frac{\varphi(u)}{z-u} du.
	\]
	In particular, for fixed $k \neq 0$, the dispersion relation $D(\lambda, k)$ is analytic in $\{ \Re \lambda >0 \}$ and admits a continuous extension up to the imaginary axis $\{ \Re \lambda =0 \}$.
\end{lemma}

\begin{proof}
	Refer to \cite[Lemma~2.3]{Nguyen2025}.
\end{proof}

	Note that the dispersion relation can also be written as
	\begin{align}
		D(\lambda, k)
		&= 1 + i \widehat{w}(k) \int_{\mathbb{R}^d} \int_{0}^{\infty} e^{-(\lambda - 2ik\cdot p + i |k|^2)t} (f(|p|^2) - f(|k-p|^2) )\, dt  dp \nonumber \\
		&= 1 + i \widehat{w}(k) \int_{0}^{\infty} e^{-\lambda t} (e^{-i|k|^2 t} - e^{i|k|^2 t}) \int_{\mathbb{R}^d} e^{2itk \cdot p } f(|p|^2) \, dp dt \nonumber \\
		&= 1 + 2 \widehat{w}(k) \int_{0}^{\infty} e^{-\lambda t} \sin(t|k|^2) \widehat{\varphi}(2t|k|) \, dt. \label{eq:disp-varphi}
	\end{align}
	
	Meanwhile, the limit of $D(\lambda, k)$ as $k \rightarrow 0$ is not well-defined at $\lambda = i \tau$ for $\tau \in \mathbb{R}$, so we introduce
	\[
		\widetilde{D}(\tilde{\lambda},k) : = D(\lambda, k),
	\]
	where $\tilde{\lambda} = \lambda / |k|$.  From \eqref{eq:disp-varphi}, we obtain for $k \neq 0 $ that
	\begin{equation*}
	\begin{split}
		\widetilde{D}(\tilde{\lambda}, k) 
		= 1 + 2 \widehat{w}(k) \int_{0}^{\infty} e^{-\tilde{\lambda} |k| t} \sin(t|k|^2) \widehat{\varphi}(2t|k|) \, dt 
		= 1 + \widehat{w}(k) \int_{0}^{\infty} e^{-\tilde{\lambda} t /2} \frac{\sin(t|k|/2)}{|k|} \widehat{\varphi}(t) \, dt,
	\end{split}
	\end{equation*}
	We define
	\begin{align}
		\widetilde{D}(\tilde{\lambda}, 0) 
		&:= \lim_{k\rightarrow 0} \widetilde{D}(\tilde{\lambda}, k) = 1 + \frac{\widehat{w}(0)}{2} \int_{0}^{\infty} e^{-\tilde{\lambda} t /2} t \widehat{\varphi}(t) \, dt \nonumber \\
		&= 1 + \frac{\widehat{w}(0)}{2} \int_{\mathbb{R}} \varphi(u) \int_{0}^{\infty} t e^{(iu - \tilde{\lambda}/2)t} \, dt du \nonumber \\
		&= 1 - \frac{\widehat{w}(0)}{2} \int_{\mathbb{R}} \frac{\varphi(u)}{( -i \tilde{\lambda}/2 - u )^2} \, du \nonumber \\
		&= 1 + \frac{\widehat{w}(0)}{2} \int_{\mathbb{R}} \frac{\varphi'(u)}{- i\tilde{\lambda}/2 - u } \, du \label{eq:disp-k=0}
	\end{align}
	for $\Re \tilde{\lambda} \ge 0$.
	Then $\widetilde{D}(\tilde{\lambda}, k)$ is continuous on $\{\Re \tilde{\lambda} \ge 0\} \times \mathbb{R}^d$.
	Now we prove that $\widetilde{D}(\tilde{\lambda}, k)$ does not admit any zeros with positive real part.

\begin{proposition}\label{prop:nozero1}
	Let $f=f(-\Delta)$ be the equilibrium satisfying \ref{equi:nonneg} and $w$ be the interaction potential satisfying \ref{assumption:potential} and \ref{assumption:potential-fourier}. For $k \in \mathbb{R}^d$, there are no zeros of $\widetilde{D}(\tilde{\lambda}, k)$ in $\{ \Re \tilde{\lambda} > 0\}$.
\end{proposition}
\begin{proof}
	Consider the case when $k\neq 0$ first. As \eqref{def:Dlambda1} holds, it is equivalent to show that the existence of solution to
	\begin{equation}\label{eq:positive-Relambda}
		\mathcal{H}(z) - \mathcal{H}(z+|k|) = \frac{2|k|}{\widehat{w}(k)}
	\end{equation}
	in $\{ \Im z < 0 \}$ yields a contradiction. Note that
	\begin{equation*}
		\mathcal{H}(z) = \int_{\mathbb{R}} \frac{\varphi(u)}{z-u} du
		=  \int_{\mathbb{R}} \frac{(\Re z - u) \varphi(u)}{|z-u|^2} du - i \Im z \int_{\mathbb{R}} \frac{\varphi(u)}{|z-u|^2} du.
	\end{equation*}
	It follows that \eqref{eq:positive-Relambda} is equivalent to the following system of equations
	\begin{equation}\label{eq:positive-Relambda2}
		\int_{\mathbb{R}} \frac{\varphi(u)}{|z-u|^2} du - \int_{\mathbb{R}} \frac{\varphi(u)}{|z+|k|-u|^2} du =0
	\end{equation}
	and
	\begin{equation}\label{eq:positive-Relambda3}
		- \int_{\mathbb{R}} \frac{u \varphi(u)}{|z-u|^2} du - \int_{\mathbb{R}} \frac{(|k|-u)\varphi(u)}{|z+|k|-u|^2} du = \frac{2|k|}{\widehat{w}(k)}.
	\end{equation}
	by reading off the real and imaginary part separately.
	From \eqref{eq:positive-Relambda2}, it can be shown that 
	\begin{equation}\label{eq:positive-Relambda4}
		\int_{\RR} \frac{\varphi(u - \frac{|k|}{2}) - \varphi(u + \frac{|k|}{2})}{|z + \frac{|k|}{2} - u|^2} \, du  = 0. 
	\end{equation}
	On the other hand, we obtain from \eqref{eq:positive-Relambda3} that
	\[
	 	- \int_{\RR} \frac{(u- \frac{|k|}{2}) (\varphi(u - \frac{|k|}{2}) - \varphi(u + \frac{|k|}{2}))}{|z+ \frac{|k|}{2} - u|^2} \, du = \frac{2|k|}{\widehat{w}(k)}.
	\]
	Using \eqref{eq:positive-Relambda4}, it follows that
	\[
		- \int_{\RR} \frac{u(\varphi(u - \frac{|k|}{2}) - \varphi(u + \frac{|k|}{2}))}{|z+ \frac{|k|}{2} - u|^2} \, du = \frac{2|k|}{\widehat{w}(k)},
	\]
	which is a contradiction as the left hand side is nonpositive, upon using the fact that $\varphi$ is even, $\varphi'(u) < 0$ for $ 0 < u < \Upsilon$, and $\widehat{w}(k) \ge 0$.

	Next, when $k=0$, it suffices to show that
	\[
		\int_{\mathbb{R}} \frac{\varphi'(u)}{z - u} \, du = - \frac{2}{\widehat{w}(0)}
	\]
	does not have any solution in $\{\Im z < 0 \}$, in view of \eqref{eq:disp-k=0}. Suppose it admits a solution in $\{ \Im z <0 \}$. Then there holds
	\[
		\int_{\mathbb{R}} \frac{(\Re z -u) \varphi'(u)}{|z-u|^2} \, du - i \Im z \int_{\mathbb{R}} \frac{\varphi'(u)}{|z-u|^2} \, du = -\frac{2}{\widehat{w}(0)}.
	\]
	Comparing the imaginary part of both sides gives
	\[
		\int_{\mathbb{R}} \frac{\varphi'(u)}{|z-u|^2} \, du =0,
	\]
	which implies
	\[
		 \int_{\mathbb{R}} \frac{u \varphi'(u)}{|z-u|^2} \, du = \frac{2}{\widehat{w}(0)}.
	\]
	This is impossible given that $\varphi'$ is odd, $\varphi'(u) < 0$ for $0< u< \Upsilon$, and $\widehat{w}(0) \ge 0$.

\end{proof}
\begin{proposition}\label{prop:nozero2}
	Let $f=f(-\Delta)$ be the equilibrium satisfying \ref{equi:nonneg} and $w$ be the interaction potential satisfying \ref{assumption:potential} and \ref{assumption:potential-fourier}. For $k \in \mathbb{R}^d$, there are no zeros of $\widetilde{D}(\tilde{\lambda}, k)$ of the form $\tilde{\lambda}= i \tilde{\tau}$ with $|\tilde{\tau}| < 2\Upsilon + |k|$.
\end{proposition}

\begin{proof}
	For $k \neq 0$ and $\tilde{\lambda} = \tilde{\gamma} +i\tilde{\tau}$ with $|\tilde{\tau}| < 2 \Upsilon +|k|$, the dispersion relation is computed as
	\[
	\widetilde{D}(\tilde{\gamma} +i\tilde{\tau}, \, k)
	= 1 + \frac{\widehat{w}(k)}{2|k|} 
	\left[ \mathcal{H} \left( \frac{-i\tilde{\gamma}+\tilde{\tau} + |k|}{2} \right) - \mathcal{H} \left( \frac{-i\tilde{\gamma}+\tilde{\tau} - |k|}{2} \right)\right]
	\]
	Notice that $\frac{\tilde{\tau} \pm |k|}{2}$ might be contained in the support of $\varphi(u)$ when $|\tilde{\tau}| < 2\Upsilon + |k|$. Hence we use the Plemelj's formula to get
	\[
	\mathcal{H} \left( \frac{\tilde{\tau} \pm |k|}{2} \right)
	= \lim_{\tilde{\gamma} \rightarrow 0} \mathcal{H} \left(\frac{-i\tilde{\gamma} +\tilde{\tau} \pm |k|}{2} \right)
	= PV \int_{\{ |u| < \Upsilon \}} \frac{\varphi(u)}{ \frac{\tilde{\tau} \pm|k|}{2}  - u } \, du + i \pi \varphi \left(\frac{\tilde{\tau} \pm |k|}{2} \right)
	\]
	where $PV$ denotes the Cauchy principal value. Thus
	\begin{equation}\label{form-Ditau2}
		\begin{aligned}	
			\widetilde{D}(i\tilde \tau,k) 
			= 1 &+ \frac{\widehat{w}(k)}{2|k|} 
			\left[ PV \int_{\{|u|<\Upsilon\}} \frac{\varphi(u)}{\frac{\tilde\tau + |k|}{2} -u } du 
			- PV \int_{\{|u|<\Upsilon\}} \frac{\varphi(u)}{\frac{\tilde\tau - |k|}{2} - u} du \right]\\[3pt]
			& + \frac{i\pi\widehat{w}(k)}{2|k|} 
			\left[ \varphi \left( \frac{\tilde\tau + |k|}{2} \right) - \varphi \left(\frac{\tilde\tau - |k|}{2} \right) \right].
		\end{aligned}
	\end{equation}
	We can derive a lower bound on $\widetilde{D}(i\tilde{\tau}, k)$ by considering its imaginary part. Namely,
	\begin{equation*}
		\begin{aligned}
			|\widetilde{D}(i\tilde{\tau},k)  |
			&\ge \frac{\pi\widehat{w}(k)}{2|k|} 
			\left|\varphi \left( \frac{\tilde\tau + |k|}{2} \right) - \varphi \left(\frac{\tilde\tau - |k|}{2} \right) \right|
			\ge \frac{\pi\widehat{w}(k)}{2|k|} 
			\left|\int_{\frac{\tilde\tau - |k|}{2} }^{\frac{\tilde\tau + |k|}{2} } \varphi'(x)\; dx \right|.
		\end{aligned}
	\end{equation*}
	The above never vanishes for $0<|\tilde \tau| <2\Upsilon + |k|$, since 
	\[
	\left(\frac{\tilde\tau - |k|}{2},\frac{\tilde\tau + |k|}{2} \right) \cap (-\Upsilon, \Upsilon)
	\]
	is nonempty and nonsymmetric. On the other hand, at $\tilde \tau =0$, observe that $\widetilde{D}(0,k)$ is real-valued and
	\begin{align*}
		\widetilde{D}(0,k)
		&= 1 + \frac{\hat{w}(k)}{2|k|} 
		\left[PV \int_{\{|u|<\Upsilon\}} \frac{\varphi(u)}{\frac{|k|}{2} - u} du 
		- PV \int_{\{|u|<\Upsilon\}} \frac{\varphi(u)}{-\frac{|k|}{2}-u} du \right].
	\end{align*}
	When $|k|\ge 2\Upsilon$, the $PV$ integrals become the usual integration, giving 
	\begin{align*}
		\widetilde{D}(0,k)
		&= 1 + \frac{\hat{w}(k)}{2} 
		\int_{\{|u|<\Upsilon\}} \frac{\varphi(u)}{\frac{|k|^2}{4} - u^2} du \ge 1 .
	\end{align*}
	When $0 < |k|\le 2\Upsilon$, we use the definition to compute the principal values as
	\begin{align*}
		PV \int_{\{ |u| < \Upsilon \}} \frac{\varphi(u)}{\frac{|k|}{2} - u} \, du 
		&= \lim_{ \epsilon \rightarrow 0^{+}} \left[ \int_{-\Upsilon}^{\frac{|k|}{2} - \epsilon} \frac{\varphi(u)}{\frac{|k|}{2} - u} \, du + \int_{\frac{|k|}{2} + \epsilon}^{\Upsilon} \frac{\varphi(u)}{\frac{|k|}{2} - u} \, du  \right] \\[3pt]
		&= 		\int_0^{\Upsilon - \frac{|k|}{2}} \frac{\varphi(u -\frac{|k|}{2} ) - \varphi( u +\frac{|k|}{2})}{u} \, du + \int_{\Upsilon - \frac{|k|}{2}}^{\Upsilon + \frac{|k|}{2}} \frac{\varphi( u-\frac{|k|}{2} )}{u} \, du.
	\end{align*}
	Recall that $\varphi(u)$ is even in $u$ and $\varphi'(u) < 0$ for $0 < u < \Upsilon$, see Lemma \ref{lem:varphi}. In particular, the integrands in the above integrals are nonnegative. Therefore, we obtain 
	\begin{align*}
		\widetilde{D}(0,k)
		&= 1 + \frac{\widehat{w}(k)}{|k|}  PV \int_{\{|u|<\Upsilon\}} \frac{\varphi(u)}{ \frac{|k|}{2} - u} du
		\ge 1 + \frac{\widehat{w}(k)}{|k|}  \int_{0}^{\Upsilon - \frac{|k|}{2}} \frac{\varphi(u -\frac{|k|}{2} ) - \varphi( u +\frac{|k|}{2})}{u} \, du
	\end{align*}
	which in particular proves $\widetilde{D}(0,k)\ge 1$. 
	
	Next, we consider the case when $k=0$. For $|\tilde{\tau}| < 2\Upsilon$, we obtain from \eqref{eq:disp-k=0} and Plemelj's formula that
	\[
	 	\widetilde{D}( i \tilde{\tau}, 0)
	 	= \lim_{ \tilde{\gamma} \rightarrow 0^{+}} \widetilde{D} (\tilde{\gamma} + i \tilde{\tau}, 0)
	 	= 1 + \frac{\widehat{w}(0)}{2} PV \int_{\{ |u| < \Upsilon \}} \frac{\varphi'(u)}{\frac{\tilde{\tau}}{2} - u} \, du + \frac{i\pi}{2} \widehat{w}(0) \varphi' (\frac{\tilde{\tau}}{2}).
	\]
	The imaginary part is nonzero for $0 < |\tilde{\tau}| < 2\Upsilon$ given that $\varphi'(u) < 0$ for $0 < u < \Upsilon$. Moreover, at $\tilde{\tau} = 0$, we obtain that $\widetilde{D}(0,0)$ is real-valued and
	\[
		\widetilde{D}(0, 0) = 1 - \frac{\widehat{w}(0)}{2} PV \int_{\{|u|<\Upsilon\}} \frac{\varphi'(u)}{u} \, du  \ge 1,
	\]
	which completes the proof of Proposition \ref{prop:nozero2}.
\end{proof}

The above proposition proves that $\widetilde{D}(\lambda, k)$ has no pure imaginary zeros when $\Upsilon =\infty$. On the other hand, when $\Upsilon < \infty$, it is still necessary to study the region $\{ \tilde{\lambda} = i \tilde{\tau}, \   |\tilde{\tau}| \ge 2\Upsilon + |k| \}$. We emphasize that \ref{assumption:stability} is equivalent to the nonexistence of oscillatory modes in this region.

\begin{proposition}\label{prop:nozero3}
	Let $f=f(-\Delta)$ be the equilibrium satisfying \ref{equi:nonneg} and $w$ be the interaction potential satisfying \ref{assumption:potential}  and \ref{assumption:potential-fourier}. Then there are no zeros of $\widetilde{D}(\tilde{\lambda}, k)$ of the form $\tilde{\lambda} = i \tilde{\tau}$ with $|\tilde{\tau}| \ge 2  \Upsilon + |k|$ if and only if \ref{assumption:stability} holds.
\end{proposition}

\begin{proof}
	It is sufficient to consider the case when $\Upsilon < \infty$. Assume that $\tilde{\lambda} = i\tilde{\tau}$ with $|\tilde{\tau}| \ge 2 \Upsilon +|k|$. Then we have $| \frac{-\tilde{\tau} \pm |k|}{2} | \ge \Upsilon$ so that 
	\begin{equation}
		\begin{aligned}\label{eq:Displargetau}
			\widetilde{D}(i \tilde{\tau}, k)
			&= 1 - \frac{\widehat{w}(k)}{2}
			\int_{ \{|u| < \Upsilon \}} \frac{\varphi(u)}{ \left( \frac{\tilde{\tau} -|k|}{2} - u \right) \left( \frac{\tilde{\tau} +|k|}{2} - u \right)} \, du.
		\end{aligned}
	\end{equation}
	As $\varphi(u)$ is even in $u$, the dispersion relation $\widetilde{D}(i\tilde{\tau}, k)$ is real-valued and even in $\tilde{\tau}$. Hence it is enough to consider the case when $\tilde{\tau} \ge 2\Upsilon + |k|$. Define 
	\begin{equation}
		\Phi(|k|) = \widetilde{D}(i (2\Upsilon +|k|) , k) = 1- \frac{\widehat{w}(k)}{2} \int_{\{ |u| <\Upsilon \}}  \frac{\varphi(u)}{(\Upsilon - u)(\Upsilon +|k| - u)} \, du,
	\end{equation}
	which is a real-valued function on $[0, \infty]$. Observe that $\Phi(\infty) = 1$ and
	\[
	\Phi(0) =  1 - \frac{\widehat{w}(0)}{2} \int_{\{ |u| < \Upsilon \}} \frac{\varphi(u)}{(\Upsilon - u)^2} \, du <1.
	\]
	Note that \ref{assumption:stability} holds if and only if $\Phi(0) >0$. Meanwhile, $\Phi$ is a function from $[0,\infty]$ onto $[\Phi(0), 1]$ as
	\[
	\Phi'(|k|) 
	=  -\frac{\widehat{w}'(k)}{2} \int_{ \{ |u| < \Upsilon \}} \frac{\varphi(u)}{(\Upsilon-u)(\Upsilon+|k|-u)} \, du 
	+ \frac{\widehat{w}(k)}{2} \int_{ \{|u|<\Upsilon \}} \frac{\varphi(u)}{(\Upsilon-u)(\Upsilon+|k|-u)^2} \, du
	\]
	is nonnegative. Also, it can be computed that 
	\[
		\partial_{\tilde{\tau}} \widetilde{D}(i\tilde{\tau}, k) = \frac{\widehat{w}(k)}{2} \int_{\{ |u| < \Upsilon \}} \frac{( \frac{\tilde{\tau}}{2} - u) \varphi(u)}{\left[ ( \frac{\tilde{\tau}}{2} - u)^2 - \frac{|k|^2}{4} \right]^2 } \, du \ge 0
	\]
	for $\tilde{\tau} > 2\Upsilon + |k|$. Hence, for a fixed $k$, the dispersion relation $\widetilde{D}(i\tilde{\tau}, k)$ is a function from $[2\Upsilon + |k| ,\infty]$ onto $[\Phi(|k|) ,1]$. This proves that $\widetilde{D}(i\tilde{\tau} ,k) > 0$ for all $k \in \mathbb{R}^d$ and $|\tilde{\tau}| \ge 2\Upsilon + |k|$ if and only if $\Phi(0) >0$.
\end{proof}

Collecting all the results, we get the stability criterion using $\langle \cdot \rangle \widehat{\varphi} \in L^1$ as follows.
\begin{proof}[Proof of Theorem \ref{thm:Penrose-Lindhard}]
	Recall that
	\begin{equation*}
	\begin{split}
		\widetilde{D}(\tilde{\lambda}, k) 
		= 1 + \widehat{w}(k) \int_{0}^{\infty} e^{-\tilde{\lambda} t /2} \frac{\sin(t|k|/2)}{|k|} \widehat{\varphi}(t) \, dt,
	\end{split}
	\end{equation*}
	Noting that $\widehat{\varphi} \in L^{1}$, we get
	\[
	|\widetilde{D}(\tilde{\lambda}, k) - 1|  \lesssim |k|^{-1} \widehat{w}(k) \int_{0}^{\infty} | \widehat{\varphi}(t) | \, dt \lesssim 	|k|^{-1} \widehat{w}(k),
	\]
	which tends to zero as $|k| \rightarrow \infty$, uniformly in $\Re \tilde{\lambda} \ge 0$. Also, we obtain for each $k \in \mathbb{R}^d$ that
	\[
		|\widetilde{D}(\tilde{\lambda}, k) - 1 | \lesssim \widehat{w}(k) \left|\int_{0}^{\infty} e^{-\tilde{\lambda} t /2} \frac{\sin(t|k|/2)}{|k|} \widehat{\varphi}(t) \, dt \right|,
	\]
	which tends to zero as $|\tilde{\lambda}| \rightarrow \infty$ with $\Re \tilde{\lambda} \ge 0$, given that $|\cdot| \widehat{\varphi} \in L^1$.
	Hence \eqref{eq:Penrose-Lindhard} holds for large $|k|$ and large $|\tilde{\lambda}|$, respectively.
	Now it suffices to consider the case when $|k| \lesssim 1$ and $|\tilde{\lambda}| \lesssim 1$ with $\Re \tilde{\lambda} \ge 0$. By compactness, it is sufficient to check $\widetilde{D}(\tilde{\lambda}, k) \neq 0$ for $k \in \mathbb{R}^d$ and $\Re \tilde{\lambda} \ge 0$. This is already proved in Propositions \ref{prop:nozero1}, \ref{prop:nozero2}, and \ref{prop:nozero3}. This concludes the proof of Theorem \ref{thm:Penrose-Lindhard}.
\end{proof}

\subsection*{Acknowledgement}
The author would like to thank Toan T. Nguyen for many helpful discussions.
The research is supported in part by the NSF under grants DMS-2054726 and DMS-2349981. 
The author has no conflict of interest to declare that are relevant to this article.
\bibliographystyle{plain}

\end{document}